\documentclass[reqno,12pt]{amsart}
\usepackage{amscd,amssymb,verbatim,array}
\usepackage{hyperref}
\usepackage{color}
\usepackage{mathrsfs}

\usepackage[margin=2.5cm]{geometry}
\numberwithin{equation}{section} 
\theoremstyle{plain}

\newcommand{\field}[1]{\mathbb{#1}}

\newcommand{\R}{\field{R}}
\newcommand{\C}{\field{C}}
\newcommand{\N}{\field{N}}

\newcommand{\ddbar}{\overline\partial}
\newcommand{\pr}{\partial}
\newcommand{\ol}{\overline}

\newcommand{\norm}[1]{\left\Vert#1\right\Vert}
\newcommand{\set}[1]{\left\{#1\right\}}
\newcommand{\To}{\rightarrow}

 \def\cC{\mathscr{C}}
\def\cL{\mathscr{L}}
\DeclareMathOperator{\Dom}{Dom}
\DeclareMathOperator\supp{supp}

\newtheorem{theorem}{Theorem}[section]
\newtheorem{lemma}[theorem]{Lemma}

\newtheorem{corollary}[theorem]{Corollary}
\theoremstyle{definition}
\newtheorem{definition}[theorem]{Definition}

\theoremstyle{definition}

\theoremstyle{remark}
\newtheorem{remark}[theorem]{Remark}

\numberwithin{equation}{section}

\newcommand{\abs}[1]{\lvert#1\rvert}

\begin{document}
%---

\title[]
{Semi-classical spectral asymptotics of Toeplitz operators on strictly pseudodonvex
domains}

\author[Chin-Yu Hsiao]{Chin-Yu Hsiao}
\address{Institute of Mathematics, Academia Sinica, Astronomy-Mathematics Building,
\newline\mbox{\quad}\,No.\ 1, Sec.\ 4, Roosevelt Road, Taipei 10617, Taiwan}
\thanks{}
\email{chsiao@math.sinica.edu.tw or chinyu.hsiao@gmail.com}
\thanks{Partially supported by the Taiwan National Science and 
Technology Council projects 108-2115-M-001-012-MY5, 
109-2923-M-001-010-MY4 and Simons Visiting Professoship of 
the Mathematisches Forschungsinstitut Oberwolfach. }
\author[George Marinescu]{George Marinescu}
\address{Universit{\"a}t zu K{\"o}ln,  Mathematisches Institut,
		Weyertal 86-90, 50931 K{\"o}ln, Germany
		\newline\mbox{\quad}\,Institute of Mathematics `Simion Stoilow', 
		Romanian Academy, Bucharest, Romania}
\thanks{Partially supported by the DFG funded projects
SFB/TRR 191 ``Symplectic Structures in Geometry, Algebra and Dynamics"
(Project-ID 281071066-TRR 191), 
and the ANR-DFG project QuaSiDy (Project-ID 490843120)}

\begin{abstract}
On a relatively compact strictly pseudoconvex domain with smooth boundary
in a complex manifold of dimension $n$
we consider a Toeplitz operator $T_R$ with symbol 
a Reeb-like vector field $R$ near the boundary.
We show that the kernel of a weighted spectral projection $\chi(k^{-1}T_R)$,
where $\chi$ is a cut-off function with compact support in the positive real line,
is a semi-classical Fourier integral operator with complex phase, hence
admits a full asymptotic expansion as $k\to+\infty$. More precisely,
the restriction to the diagonal $\chi(k^{-1}T_R)(x,x)$ decays at the
rate $O(k^{-\infty})$ in the interior and has an asymptotic expansion
on the boundary with leading term of order $k^{n+1}$
expressed in terms of the Levi form and the pairing of the contact form
with the vector field $R$.
\end{abstract}

\maketitle \tableofcontents

\section{Introduction}\label{s-gue230520yyd}
Since the introduction of the Bergman kernel in \cite{Berg32}, and the subsequent 
groundbreaking work by H\"ormander \cite{Hor65},  Fefferman \cite{Fer74}, 
and Boutet de Monvel and Sj\"ostrand \cite{BS75}, the study of the 
Bergman kernel has been a central subject in several complex variables
and complex geometry. 

Let $M$ be a relatively compact strictly pseudoconvex domain 
with smooth boundary in 
a complex manifold $M'$ and let
$B:L^2(M)\To H^0_{(2)}(M)$ be the Bergman projection, that is,
the orthogonal projection from the space
of square-integrable functions $L^2(M)$
onto the space of $L^2$ holomorphic functions on $M$. 
The Bergman kernel $B(x,y)$ is the Schwartz kernel of $B$.
Fefferman~\cite{Fer74} obtained the complete asymptotic expansion of
the diagonal Bergman kernel $B(x,x)$ at the boundary.
Subsequently, Boutet de Monvel-Sj\"{o}strand~\cite{BS75} 
described the singularity of the full Bergman kernel $B(x,y)$
by showing that it is a Fourier integral operator with
complex phase.
They also obtained in~\cite{BS75} a full asymptotic expansion for 
the Szeg\H{o} projection on a strictly pseudoconvex CR manifold. 
All the results mentioned above are about microlocal behavior 
of the Szeg\H{o} and Bergman kernels. 
Some of these results were recently extended to weakly pseudoconvex domains
of finite type in $\C^2$ in \cite{HS22}.
The structure of the Szeg\H{o} projector also plays an important 
role in the quantization of CR manifolds \cite{HMM}.

On the other hand, 
semi-classical analysis plays an important role in 
modern complex geometry. For example, 
we can study many important problems in complex geometry 
by using semi-classical 
Bergman kernel asymptotics~\cite{DLM06, HM14CAG, HM17CPDE, MM07}. 
Therefore,  we believe that it is 
important to study classical several complex variables from semi-classical viewpoint. 
To this end, it is important to have semiclassical versions of the 
Boutet de Monvel-Sj\"ostrand's and Fefferman's results 
on strictly pseudoconvex CR manifolds and on complex 
manifolds with strictly pseudoconvex boundary. 
Recently,  
we obtained jointly with Herrmann and Shen~\cite{HHMS23} 
a semi-classical version of the 
Boutet de Monvel and Sj\"ostrand's result on a strictly peudoconvex 
CR manifold and as applications, we established Kodaira type embedding 
theorem and Tian type theorem on a strictly pseudoconvex CR manifold. 

It is natural to establish similar results as in~\cite{HHMS23} for complex manifolds with boundary. %This is the motivation of this work. 
In this paper, we consider 
the operator $\chi_k(T_R)$ constructed by functional
calculus, where $\chi_k(\lambda)=\chi(k^{-1}\lambda)$ is a rescaled cut-off function $\chi$ 
with compact support in the positive real line, $k\in\R_+$ is a semi-classical parameter,
and $T_R$ is the Toeplitz operator on the domain $M$ associated with 
a first-order differential operator given by
a Reeb-like vector field from in the neighborhood of $X$.
We show that $\chi_k(T_R)$
admits a full asymptotic expansion as $k\to+\infty$. This result can be seen as a 
semi-classical version of the 
Boutet de Monvel-Sj\"ostrand's and Fefferman's results on complex manifolds with boundary. 

We now formulate our main result. 
We refer the reader to Section~\ref{s-g2204301052} 
for the notations used here. 
Let $(M',J)$ be a complex manifold of dimension $n$ with complex structure $J$.
We fix a Hermitian metric $\Theta$ on $M'$ and 
let $g^{TM'}=\Theta(\cdot, J\cdot)$ be the Riemannian metric
on $TM'$ associated to $\Theta$ and let $dv_{M'}$ be its volume form.
We denote by $\langle\,\cdot\,|\,\cdot\,\rangle$ the
pointwise Hermitian product induced by $g^{TM'}$
on the fibers of $\C TM'$ and by duality on 
$\C T^{*}M'$. 

Let $M$ be a relatively compact open subset in $M$ with smooth boundary.
We set $X=\partial{M}$. 
We assume throughout the paper that $M$ is strictly pseudoconvex.
Let $\rho\in\cC^\infty(M',\mathbb R)$ be a defining function of $M$
(cf.\ \eqref{eq:def_funct}), let $\cL_x=\cL_x(\rho)$ 
be the the Levi form associated to $\rho$
at $x\in X$ (cf.\ \eqref{eq:2.12b}) and let $\det(\cL_x)$ 
be the determinant of the Levi form (cf.\ \eqref{eq_detLevi}).
We consider the 1-form $\omega_0=-d\rho\circ J=i(\ddbar\phi-\partial\phi)$
and we fix the contact form
$\omega_0|_{TX}=2i\ddbar\rho|_{TX}=-2i\partial\rho|_{TX}$
on $X$ (cf.\ \eqref{eq:omega_0}-\eqref{eq:2.12b}).

Let $(\,\cdot\,|\,\cdot\,)_M$, $(\,\cdot\,|\,\cdot\,)_{M'}$ be the $L^2$ inner products 
on $\cC^\infty(\ol M)$, $\cC^\infty_c(M')$ induced by the given Hermitian metric 
$\langle\,\cdot\,|\,\cdot\rangle$ respectively (see \eqref{e-gue190312}). 
Let $L^2(M)$ be the completion of $\cC^\infty(\ol M)$ 
with respect to $(\,\cdot\,|\,\cdot\,)_M$. 
Let 
$H^0(\ol{M}):=\set{u\in\cC^\infty(\ol M);\, \ddbar u=0}$,
where $\ddbar:\cC^\infty(M')\To\Omega^{0,1}(M')$ 
denotes the standard Cauchy-Riemann 
operator on $M'$.
Let $H^0_{(2)}(M)$ be the completion of $H^0(\ol M)$ with respect to 
$(\,\cdot\,|\,\cdot,)_M$.

We denote by $\nabla\rho$ the gredient of $\rho$
with respect to the Riemannian metric $g^{TM'}$.
We consider the vector field
$T=\alpha J\big(\nabla\rho\big)+Z$ on $M'$,
where $\alpha\in\cC^\infty(M')$, $a|_X>0$,
and $Z\in\cC^\infty(M',TM')$, $Z|_X\in\cC^\infty(X,HX)$,
%defined in a neighbourhood of $X$, 
cf.\ \eqref{e-gue190312scdqI}-\eqref{e-gue190312scdqII}.
Let $R$ be a formally selfadjoint first order partial differential 
operator on $M'$, 
given near $X$ by
$R=\frac{1}{2}((-iT)+(-iT)^*)$.

Since $M$ is strictly pseudoconvex we have by~\cite{BS75,Fer74} that
the Bergman projection maps the space $\cC^\infty(\ol M)$ of
smooth functions up to the boundary into itself,
$B: \cC^\infty(\ol M)\To\cC^\infty(\ol M)$.
Let 
\[T_R:=BRB: \cC^\infty(\ol M)\To\cC^\infty(\ol M).\] 
We extend $T_R$ to $L^2(M)$: 
\begin{equation}\label{e-gue230528yydq}\begin{split}
&T_R:\Dom(T_R)\subset L^2(M)\To L^2(M),\\
&\Dom(T_R)=\set{u\in L^2(M);\, BRBu\in L^2(M)},
\end{split}\end{equation}
where $BRBu$ is defined in the sense of distributions on $M$. 
In Theorem~\ref{t-gue230525yyd}, we will show that $T_R$ is self-adjoint. 
We consider a function
\begin{equation}\label{e-gue230527yyd0}
\chi\in\cC^\infty_c(\mathbb{R}_+,\mathbb R),
\end{equation}
and set for $k>0$,
\begin{equation}\label{e-gue230527yyd}
\chi_k\in\cC^\infty_c(\mathbb{R}_+,\mathbb R),\quad
\chi_{k}(\lambda):=\chi(k^{-1}\lambda).
\end{equation}
We let
\begin{equation}\label{eq:chi-k}
\quad \chi_k(T_R):L^2(M)\To L^2(M),
\end{equation}
be obtained by functional calculus 
of $T_R$ and let 
$\chi_k(T_R)(\cdot\,,\cdot)\in\mathscr D'(M\times M)$ 
be the distribution kernel of $\chi_k(T_R)$. 
We will show that 
$\chi_k(T_R)(\cdot\,,\cdot)\in\cC^\infty(\ol M\times\ol M)$
(cf.\ Corollary~\ref{t-gue230527ycd}). 
We consider a function $\chi$ with support in $(0,+\infty)$
in order to avoid that the spectral operator $\chi_k(T_R)$
takes into account the zero eigenvalue of $T_R$, whose
eigenspace contains the kernel of $B$. With this choice
the image of $\chi_k(T_R)$ is contained in $H^0_{(2)}(M)$.

The main result of this paper is the following.
%---
\begin{theorem}\label{t-gue230528yydmp}
Let $M$ be a relatively compact strictly pseudoconvex 
domain with smooth boundary $X$ of a
complex manifold $M'$ of dimension $n$. 
Let $T_R:\Dom(T_R)\subset L^2(M)\To L^2(M)$ be the Toeplitz
operator \eqref{e-gue230528yydq} and let $\chi_k(T_R)$
be as in \eqref{eq:chi-k}.
Then the following assertion hold:

(i) For any $\tau, \hat\tau\in\cC^\infty(\ol M)$ with
$\supp\tau\cap\supp\hat\tau=\emptyset$ we have 
\begin{equation}\label{e-gue230527yydzz}
\tau\chi_k(T_R)\hat\tau=O(k^{-\infty})\quad
\text{on $\ol M\times\ol M$}. 
\end{equation}

(ii) %For any $p\in M$, any open neighborhood $U\subset M$ of $p$ and
%$\tau\in\cC^\infty_c(U)$ we have
For any $\tau\in\cC^\infty_c(M)$ we have
\begin{equation}\label{e-gue230528ycdz}
\tau\chi_k(T_R)=O(k^{-\infty})\quad \text{on $\ol M\times\ol M$}.
\end{equation}

(iii) For any $p\in X$ and any open local coordinate patch $U$ around $p$ in $M'$
we have 
\begin{equation}
\label{e-gue230528ycdsz}
\chi_k(T_R)(x,y)=\int_0^{+\infty} 
e^{ikt\Psi(x,y)}b(x,y,t,k)dt+O(k^{-\infty})
\quad\text{on $(U\times U)\cap(\ol M\times\ol M)$},
\end{equation}
where $b(x,y,t,k)\in S^{n+1}_{\mathrm{loc}}
(1;((U\times U)\cap(\ol M\times\ol M))\times{\mathbb R}_+)$,
\begin{equation}
\label{e-gue230528ycdtz}
\begin{split}
&b(x,y,t,k)\sim\sum_{j=0}^\infty b_{j}(x,y,t)k^{n+1-j}~
\mathrm{in}~S^{n+1}_{\mathrm{loc}}(1;((U\times U)
\cap(\ol M\times\ol M))\times{\mathbb R}_+),\\
&b_j(x,y,t)\in\mathscr{C}^\infty(((U\times U)\cap(\ol M\times\ol M))
\times{\mathbb R}_+),~j=0,1,2,\ldots,\\
&b_{0}(x,x,t)=\frac{1}{\pi ^{n}}
\det(\cL_x)\,\chi(t\omega_0(T(x)))t^n\not\equiv 0,\ \ x\in U\cap X,
\end{split}
\end{equation}
and for some compact interval $I\Subset\mathbb R_+$,
\begin{equation}
\begin{split}
\supp_t b(x,y,t,k),\:\:\supp_t b_j(x,y,t)\subset I,\ \ j=0,1,2,\ldots,
\end{split}
\end{equation}
and 
\begin{equation}\label{e-gue230528yydzz}
\begin{split}
&\Psi(z,w)\in\cC^\infty((U\times U)\cap(\ol M\times\ol M)),\ \ {\rm Im\,}\Psi\geq0,\\
&\Psi(z,z)=0, \ z\in U\cap X,\\
&\mbox{${\rm Im\,}\Psi(z,w)>0$ if $(z,w)\notin(U\times U)\cap(X\times X)$},\\
&d_x\Psi(x,x)=-d_y\Psi(x,x)=-2i\partial\rho(x),\ \ x\in U\cap X,\\
&\Psi|_{(U\times U)\cap(X\times X)}=\varphi_-, 
\end{split}
\end{equation}  
where $\varphi_-$ is a phase function as in
\eqref{e-gue140205IV}, cf.\
~\cite[Theorem 4.1]{HM14}.
Moreover, using the local coordinates $z=(x_1,\ldots,x_{2n-1},\rho)$
on $M'$ near $p$, where
$x=(x_1,\ldots,x_{2n-1})$ are local coordinates on $X$ near 
$p$ with $x(p)=0$, we have 
\begin{equation}\label{e-gue230319ycdaIm}
\Psi(z,w)=\Psi(x,y)-i\rho(z)(1+f(z))-
i\rho(w)(1+\overline{f(w)}\,)+O(\abs{(z,w)}^3)\:\:\text{near $(p,p)$},
\end{equation}
where $f$ is smooth near $p$ and $f=O(\abs{z})$. 
\end{theorem} 
%---
The representation \eqref{e-gue230528ycdsz}
shows that near the boundary $\chi_k(T_R)$ 
is a semi-classical Fourier integral 
operator with complex phase and
canonical relation generated by the phase $\Psi(x,y)t$. 
The integral in \eqref{e-gue230528ycdsz} 
is a smooth kernel, since $t$ runs in the bounded interval
$I$. The term $O(k^{-\infty})$ denotes a $k$-negligible
smooth kernel (cf.\ \eqref{e-gue13628III}-\eqref{e-gue13628IV}).

The idea of the proof of Theorem \ref{t-gue230528yydmp}
follows the strategy of \cite{BS75,Hsiao08}. We express
the Bergman projection in terms of the Poisson operator
%and a generalized Szeg\H{o} projector $\mathcal{S}$, that is,
and a projector $\mathcal{S}$ on a subspace of functions annihilated by 
a system of pseudo-differential operators simulating $\ddbar_b$. 
We can express in the same way a Toeplitz operator $T_R$
in terms of a Toeplitz operator on the boundary $X$,
given by 
$\mathcal{T}_{\mathcal{R}}=\mathcal{S}\mathcal{R}\mathcal{S}$.
The operator
$\mathcal{S}$ is a Fourier integral operator having a structure
similar to the Szeg\H{o} projector 
cf.\ \cite{BG81,BS75,HHMS23,Hsiao08}
and we can apply the results obtained in \cite{HHMS23}
for the asymptotics of $\chi_k(\mathcal{T}_{\mathcal{R}})$.

As a consequence we have the following asymptotics of
the kernel of $\chi_k(T_R)$ on the diagonal.
%---
\begin{corollary}\label{cor:trace}
In the situation of Theorem \ref{t-gue230528yydmp}
we have:
\begin{equation}\label{eq:tr1}
\chi_k(T_R)(z,z)=O(k^{-\infty}),\quad\text{as $k\to\infty$
on $M$.}
\end{equation}
\begin{equation}\label{eq:1.8}
\chi_k(T_R)(x,x)=\sum_{j=0}^{\infty} b_j(x)k^{n+1-j}~
\text{in $S^{n+1}_{\rm loc}(1;X)$ on $X$},
\end{equation}
where for $x\in X$ and with $b_j(x,x,t)$ as in \eqref{e-gue230528ycdtz},
\begin{equation}\label{eq:1.8a}
b_j(x)=\int_0^{+\infty}b_j(x,x,t)dt,\quad j\in\N_0,
\end{equation}
with
\begin{equation}\label{eq:1.8b}
b_0(x)=\frac{1}{\pi ^{n}}\det(\cL_x)\int_0^{+\infty}
\chi(t\omega_0(T(x)))t^ndt,
\end{equation}
Moreover, there exist $C_1,C_2>0$ 
such that for $k$ large enough 
$C_1k^n\leq\operatorname{Tr}\chi_k(T_R)\leq C_2k^n$.
\end{corollary}

\addtocontents{toc}{\protect\setcounter{tocdepth}{0}}
\section*{Acknowledgement}
Chin-Yu Hsiao would like to thank the Department
of Mathematics and Computer Science of the 
University of Cologne for the hospitality during his visits in
September-November 2022 and May 2023. 
\addtocontents{toc}{\protect\setcounter{tocdepth}{2}}
\section{Preliminaries}\label{s-g2204301052} 
\subsection{Notions from microlocal and semi-classical analysis} \label{s-ssna} 

We shall use the following notations: 
$\mathbb N=\set{1,2,\ldots}$, $\mathbb N_0=\mathbb N\cup\set{0}$, $\mathbb R$ 
is the set of real numbers, $\mathbb R_+:=\set{x\in\mathbb R;\, x>0}$. 

Let $W$ be a smooth paracompact manifold.
We let $TW$ and $T^*W$ denote the tangent bundle of $W$
and the cotangent bundle of $W$ respectively.
The complexified tangent bundle of $W$ and the complexified 
cotangent bundle of $W$ are be denoted by $\C TW$
and $\C T^*W$, respectively. Write $\langle\,\cdot\,,\cdot\,\rangle$ to denote the pointwise
duality between $TW$ and $T^*W$.
We extend $\langle\,\cdot\,,\cdot\,\rangle$ bilinearly to $\C TW\times\C T^*W$.
Let $G$ be a smooth vector bundle over $W$. The fiber of $G$ at $x\in W$ will be denoted by $G_x$.
Let $Y\subset W$ be an open set. 
From now on, the spaces of distributions of $Y$ and
smooth functions of $Y$ will be denoted by $\mathscr D'(Y)$ and $\cC^\infty(Y)$ respectively.
Let $\mathscr E'(Y)$ be the subspace of $\mathscr D'(Y)$ whose elements have compact support in $Y$ and 
let $\cC^\infty_c(Y)$ be the subspace of $\cC^\infty(Y)$ whose elements have compact support in $Y$. For $m\in\mathbb R$, let $H^m(Y)$ denote the Sobolev space
of order $m$ of $Y$. Put
\begin{gather*}
H^m_{\rm loc\,}(Y)=\big\{u\in\mathscr D'(Y);\, \varphi u\in H^m(Y),
      \, \mbox{for every $\varphi\in \cC^\infty_c(Y)$}\big\}\,,\\
      H^m_{\rm comp\,}(Y)=H^m_{\rm loc}(Y)\cap\mathscr E'(Y)\,.
\end{gather*}

If
$A: \cC^\infty_c(W)\To \mathscr D'(W)$
is continuous, we write $A(x, y)$ to denote the distribution kernel of $A$.
The following two statements are equivalent
\begin{enumerate}
\item $A$ is continuous: $\mathscr E'(W)\To \cC^\infty(W)$,
\item $A(x,y)\in \cC^\infty(W\times W)$.
\end{enumerate}
If $A$ satisfies (1) or (2), we say that $A$ is smoothing on $W$. Let
$A,B: \cC^\infty_c(W)\to \mathscr D'(W)$ be continuous operators.
We write 
\begin{equation} \label{e-gue160507f}
\mbox{$A\equiv B$ (on $W\times W$)} 
\end{equation}
if $A-B$ is a smoothing operator on $W$. 

%We say that $A$ is properly supported if the restrictions of the two projections 
%$(x,y)\To x$, $(x,y)\To y$ to $\supp A(x,y)$
%are proper.

Let $H(x,y)\in\mathscr D'(W\times W)$. We write $H$ to denote the unique 
continuous operator $\cC^\infty_c(W)\To\mathscr D'(W)$ with distribution kernel $H(x,y)$. 
In this work, we identify $H$ with $H(x,y)$. 

Let $D$ be an open set of a smooth manifold $X$. 
For $0\leq\rho, \delta\leq1$, $m\in\mathbb R$, let
$$L^m_{\rho,\delta}(D)\,,\:\:
\:\:L^m_{{\rm cl\,}}(D),$$
denote the space of pseudodifferential operators on $D$ of order 
$m$ type $(\rho,\delta)$
and the 
space of classical pseudodifferential operators on $D$ of order $m$ 
 respectively. Let $W\subset\mathbb R^N$ be an open set.
For $m\in\mathbb R$, $0\leq\rho, \delta\leq1$, let 
$S^{m}_{\rho,\delta}(W\times\mathbb R^{N_1})$ be the 
H\"ormander symbol space on $W\times\mathbb R^{N_1}$ of order $m$ and 
type $(\rho,\delta)$. Let
$S^{m}_{{\rm cl\,}}(W\times\mathbb R^{N_1})$ be the 
classical symbol space on $W\times\mathbb R^{N_1}$ of order $m$.

Let $W_1$ be an open set in $\mathbb R^{N_1}$ and 
let $W_2$ be an open set in $\mathbb R^{N_2}$. 
A $k$-dependent continuous operator
$F_k:\cC^\infty_c(W_2)\To\mathscr D'(W_1)$
is called $k$-negligible on $W_1\times W_2$
if, for $k$ large enough, $F_k$
is smoothing and, for any $K\Subset W_1\times W_2$, any
multi-indices $\alpha$, $\beta$ and any $N\in\mathbb N$, 
there exists $C_{K,\alpha,\beta,N}>0$
such that
\begin{equation}\label{e-gue13628III}
\abs{\pr^\alpha_x\pr^\beta_yF_k(x, y)}\leq 
C_{K,\alpha,\beta,N}k^{-N}\:\: \text{on $K$},\ \ \text{ for } k\gg1.
\end{equation}
In that case we write
\begin{equation}\label{e-gue13628IV}
F_k(x,y)=O(k^{-\infty})\:\:\text{or}\:\: F_k=O(k^{-\infty})
\:\:\text{on $W_1\times W_2$}.
\end{equation}
If $F_k, G_k:\cC^\infty_c(W_2)\To\mathscr D'(W_1)$
are $k$-dependent continuous operators, 
we write $F_k= G_k+O(k^{-\infty})$ on $W_1\times W_2$ 
or $F_k(x,y)=G_k(x,y)+O(k^{-\infty})$ on $W_1\times W_2$ if 
$F_k-G_k=O(k^{-\infty})$ on $W_1\times W_2$. 

Let $\Omega_1$ and $\Omega_2$ be smooth manifolds. Let $F_k, G_k: \cC^\infty(\Omega_2)\To\cC^\infty(\Omega_1)$
be $k$-dependent smoothing operators. 
We write $F_k=G_k+O(k^{-\infty})$ on $\Omega_1\times\Omega_2$ 
if on every local coordinate patch $D$ of $\Omega_1$ and local coordinate patch 
$D_1$ of $\Omega_2$, $F_k=G_k+O(k^{-\infty})$ on $D\times D_1$.

We recall the definition of the semi-classical symbol spaces.

\begin{definition} \label{d-gue140826}
Let $W$ be an open set in $\mathbb R^N$. Let
\begin{equation*}
\begin{split}
&S(1;W):=\Big\{a\in\cC^\infty(W);\, \mbox{for every $\alpha\in\mathbb N^N_0$}:
\sup_{x\in W}\abs{\pr^\alpha a(x)}<\infty\Big\},\\
S^0_{{\rm loc\,}}(1;W):=&
\Big\{(a(\cdot,k))_{k\in\mathbb R};\, \mbox{for all $\alpha\in\mathbb N^N_0$,
$\chi\in\cC^\infty_c(W)$}\,:\:\sup_{k\geq1}\sup_{x\in W}
\abs{\pr^\alpha(\chi a(x,k))}<\infty\Big\}\,.
\end{split}
\end{equation*}
For $m\in\mathbb R$, let
\[
S^m_{{\rm loc}}(1):=S^m_{{\rm loc}}(1;W)
=\Big\{(a(\cdot,k))_{k\in\mathbb R};\, (k^{-m}a(\cdot,k))
\in S^0_{{\rm loc\,}}(1;W)\Big\}\,.
\]
Hence $a(\cdot,k)\in S^m_{{\rm loc}}(1;W)$ 
if for every $\alpha\in\mathbb N^N_0$ and $\chi\in\cC^\infty_c(W)$, there
exists $C_\alpha>0$ independent of $k$, 
such that $\abs{\pr^\alpha (\chi a(\cdot,k))}\leq C_\alpha k^{m}$ 
holds on $W$.

Consider a sequence $a_j\in S^{m_j}_{{\rm loc\,}}(1)$, 
$j\in\N_0$, where $m_j\searrow-\infty$,
and let $a\in S^{m_0}_{{\rm loc\,}}(1)$. We say that
\[
a(\cdot,k)\sim
\sum\limits^\infty_{j=0}a_j(\cdot,k)\:\:
\text{in $S^{k_0}_{{\rm loc\,}}(1)$},
\]
if for every
$\ell\in\N_0$, we have $a-\sum^{\ell}_{j=0}a_j\in 
S^{m_{\ell+1}}_{{\rm loc\,}}(1)$ .
For a given sequence $a_j$ as above, we can always 
find such an asymptotic sum
$a$, which is unique up to an element in
$S^{-\infty}_{{\rm loc\,}}(1)=S^{-\infty}_{{\rm loc\,}}(1;W)
:=\cap _mS^m_{{\rm loc\,}}(1)$.

Similarly, we can define $S^m_{{\rm loc\,}}(1;Y)$ 
in the standard way, where $Y$ is a smooth manifold.
\end{definition}

\subsection{Set up of complex manifolds with smooth boundary} \label{s-su}
Let $(M',J)$ be a complex manifold of dimension $n$, where
$J: TM'\To TM'$ is the complex structure of $M'$.
We fix a Hermitian metric $\Theta$ on $M'$ and 
let $g^{TM'}=\Theta(\cdot, J\cdot)$ be the Riemannian metric
on $TM'$ associated to $\Theta$ and let $dv_{M'}$ be its volume form.
We denote by $\langle\,\cdot\,|\,\cdot\,\rangle$ the
pointwise Hermitian product induced by $g^{TM'}$
on the fibers of the bundle $\Lambda^q(T^{*(0,1)}M')$
of $(0,q)$-forms for every $q\in\{0,\ldots,n\}$.
Let $\Omega^{0,q}(M')$ be the space of smooth 
$(0,q)$-forms on $M'$ and let $\Omega^{0,q}_c(M')$ 
be the subspace of $\Omega^{0,q}(M')$ 
whose elements have compact support in $M'$.
The $L^2$ inner product on $\Omega^{0,q}_c(M')$ is given by 
\begin{equation}
(\,\alpha\,|\,\beta\,)_{M'}=\int_{M'}
\langle\,\alpha\,|\,\beta\,\rangle dv_{M'}.
\end{equation}
The corresponding $L^2$ space is denoted by $L^2_{0,q}(M')$,
and we set $L^2(M')=L^2_{0,0}(M')$.

Let $M$ be a relatively compact open subset of $M'$
with smooth boundary. Hence $X:=\partial M$ is a submanifold
of $M'$ of real dimension $2n-1$. 
We denote by $HX=TX\cap J(TX)$ the complex tangent bundle
of $X$. The triple $(X,HX,J)$ forms a CR structure on $X$
and we set $T^{1,0}X:=T^{1,0}M'\cap\C TX$,
$T^{0,1}X:=\ol{T^{1,0}X}$. 
Let $\rho\in\cC^\infty(M',\mathbb R)$ be a defining function of $X$, 
that is, 
\begin{equation}\label{eq:def_funct}
M=\{x\in M':\rho(x)<0\},\quad X=\partial M=\{x\in M':\rho(x)=0\}, \quad
\text{and $d\rho\neq0$ on $X$.}
\end{equation}
From now on, we fix a defining function $\rho$ so that $|d\rho|=1$ on $X$. 
Define a real $1$-form $\omega_0$ on $M'$ by
\begin{equation}\label{eq:omega_0}
\omega_0=-d\rho\circ J.
\end{equation}
Hence
\begin{equation}\label{eq:domega_0}
\omega_0=i(\ddbar\rho-\partial\rho),\quad
d\omega_0=2i\partial\ddbar\rho.
\end{equation}
The Levi form of $\rho$ is Hermitian symmetric map $
\cL_x=\cL_x(\rho)$ given by
\begin{equation}\label{eq:2.12b}
\cL_x:T^{1,0}_xX\times T^{1,0}_xX\to\C,
\quad \cL_x(U,\ol V)=\frac{1}{2i} d\omega_0(U, \ol V)
=\partial\ddbar\rho(U, \ol V),
\quad U, V\in T^{1,0}_xX.
\end{equation}

We assume that $M$ is a strictly pseudoconvex domain, that is,
the Levi form $\cL_x$ is positive definite for every $x\in X$.
In this case the hyperplane field $HX$ is a contact
structure on $X$. Indeed, $HX=\ker(\omega_0|_{TX})$
and for every $u\in HX\setminus\{0\}$ we have 
$d\omega_0(u,Ju)=4\cL(U,\overline{U})>0$,
where $U=\frac12(u-iJu)\in T^{1,0}X$.
So $d\omega_0|_{HX}$ is symplectic, and hence $HX$ 
is a contact structure, with 
$\omega_0|_{TX}=2i\ddbar\rho|_{TX}=-2i\partial\rho|_{TX}$ 
a contact form.

We denote by $\lambda_j(x)$, $j=1,\ldots,n-1$, the eigenvalues 
of $\cL_x$ with respect to $\langle\,\cdot\,|\,\cdot\,\rangle$
(note that $T^{0,1}X$ has rank $n-1$).
The determinant of the Levi form is defined by
%---
\begin{equation}\label{eq_detLevi}
\det(\cL)_x:=\lambda_1(x)\ldots\lambda_{n-1}(x).
\end{equation}
%---

Let $\nabla\rho$ be the gradient
of $\rho$ with respect to the Riemannian metric $g^{TM'}$.
We define the vector field $T$ on $M'$ by
\begin{equation}\label{e-gue190312scdqI}
T=\alpha J\big(\nabla\rho\big)+Z\in\cC^\infty(M',TM'),
\end{equation}
where 
\begin{equation}\label{e-gue190312scdqII}
\alpha\in\cC^\infty(M'),\:\: \alpha|_X>0,\:\: 
Z\in\cC^\infty(M',TM'),\:\: Z|_X\in\cC^\infty(X,HX).
\end{equation}
The vector field $T$ does not vanish on a neighborhood of $X$.
Indeed, we have on $X$ that $\langle J(\nabla\rho),Z\rangle=
-\langle\nabla\rho,JZ\rangle=0$  hence
$|T|^2=a^2+|Z|^2>0$. 
Note also that 
\begin{equation}\label{eq:omegaT}
\omega_0(T)=-(d\rho\circ J)(\alpha J(\nabla\rho)+Z)=
\alpha d\rho(\nabla\rho)=\alpha|\nabla\rho|^2=\alpha\quad\text{on $X$},
\end{equation}
since $|\nabla\rho|=|d\rho|=1$ on $X$.
 
Let $U$ be an open set in $M'$. Let 
\[
\begin{split}
&\cC^\infty(U\cap \ol M),\ \ \mathscr D'(U\cap \ol M),\ \ \cC^\infty_c(U\cap \ol M),\ \ 
\mathscr E'(U\cap \ol M),\\ 
&H^s(U\cap \ol M),\ \ H^s_{{\rm comp\,}}(U\cap \ol M),\ \ 
H^s_{{\rm loc\,}}(U\cap \ol M),
\end{split}
\]
(where $\ s\in\mathbb R$)
denote the spaces of restrictions to $U\cap\ol M$ of elements in 
\[
\begin{split}
&\cC^\infty(U\cap M'),\ \ \mathscr D'(U\cap M'),\ \ \cC^\infty(U\cap M'),\ \ 
\mathscr E'(U\cap M'),\\  
&H^s(M'),\ \  H^s_{{\rm comp\,}}(M'),\ \  
H^s_{{\rm loc\,}}(M'),
\end{split}
\] 
respectively. Write 
\[
\begin{split}
L^2(U\cap\ol M):=&H^0(U\cap \ol M),\ \ 
L^2_{{\rm comp\,}}(U\cap\ol M):=H^0_{{\rm comp\,}}(U\cap \ol M),\\ 
&L^2_{{\rm loc\,}}(U\cap\ol M):=H^0_{{\rm loc\,}}(U\cap \ol M).
\end{split}
\] 
Let $dv_{M'}$ be the volume form on $M'$ induced 
by the Hermitian metric $\langle\,\cdot\,|\,\cdot\,\rangle$
on $\C TM'$ and 
and let $(\,\cdot\,|\,\cdot\,)_M$ and $(\,\cdot\,|\,\cdot\,)_{M'}$ 
be the inner products on $\cC^\infty(\ol M)$ and $\cC^\infty_c(M')$
defined by
\begin{equation} \label{e-gue190312}
\begin{split}
&(\,f\,|\,h\,)_M=\int_Mf\ol hdv_{M'},\ \ 
f, h\in\cC^\infty(\ol M),\\
&(\,f\,|\,h\,)_{M'}=\int_{M'}f\ol hdv_{M'},\ \ 
f, h\in\cC^\infty_c(M').
\end{split}
\end{equation}
Let $\norm{\cdot}_M$ and $\norm{\cdot}_{M'}$ be the 
corresponding norms with respect to $(\,\cdot\,|\,\cdot\,)_M$
and $(\,\cdot\,|\,\cdot\,)_{M'}$ respectively. 
Let $L^2(M)$ be the completion of $\cC^\infty(\ol M)$ 
with respect to $(\,\cdot\,|\,\cdot\,)_M$. 
We extend $(\,\cdot\,|\,\cdot\,)_M$ to $L^2(M)$
in the standard way. For $q=1,2,\ldots,n$, 
let $\Omega^{0,q}(M')$ be the space of smooth $(0,q)$ forms on $M'$ and 
let $\Omega^{0,q}_c(M')$ be the subspace of 
$\Omega^{0,q}(M'
)$ whose elements have compact support in $M'$. 
As in \eqref{e-gue190312}, let $(\,\cdot\,|\,\cdot\,)_{M'}$ 
be the $L^2$ inner product on $\Omega^{0,q}_c(M')$ 
induced by $dv_{M'}$ and $\langle\,\cdot\,|\,\cdot\,\rangle$.

The boundary $X=\partial{M}$ is a compact CR manifold 
of dimension $2n-1$ with natural CR structure 
$T^{1,0}X:=T^{1,0}M'\cap\C TX$. Let $T^{0,1}X:=\ol{T^{1,0}X}$. 
The Hermitian metric on $\C TM'$ induces a Hermitian metric
$\langle\,\cdot\,|\,\cdot\,\rangle$ on $\C TX$ and let $(\,\cdot\,|\,\cdot\,)_X$ 
be the $L^2$ inner product on 
$\cC^\infty(X)$ 
induced by $\langle\,\cdot\,|\,\cdot\,\rangle$.

Let $U$ be an open set in $M'$. Let 
$$F_1, F_2: \cC^\infty_c(U\cap M)\To\mathscr D'(U\cap M)$$ 
be continuous operators. Let 
$F_1(x,y), F_2(x,y)\in\mathscr D'((U\times U)\cap(M\times M))$ 
be the distribution kernels of $F_1$ and $F_2$ respectively. 
We write 
$$F_1\equiv F_2\!\!\mod\cC^\infty((U\times U)\cap(\ol M\times\ol M))$$ 
or $F_1(x,y)\equiv F_2(x,y)\!\!\mod\cC^\infty((U\times U)\cap(\ol M\times\ol M))$ 
if $F_1(x,y)=F_2(x,y)+r(x,y)$, where 
$r(x,y)\in\cC^\infty((U\times U)\cap(\ol M\times\ol M))$. 

Let $F_k, G_k: \cC^\infty_c(U\cap M)\To\mathscr D'(U\cap M)$ 
be $k$-dependent continuous operators. 
Let $F_k(x,y), G_k(x,y)\in\mathscr D'((U\times U)\cap(M\times M))$ 
be the distribution kernels of $F_k$ and $G_k$ respectively. 
We write 
\begin{equation}\label{e-gue190813yyd}
F_k(x,y)\equiv G_k(x,y)\!\!\mod O(k^{-\infty})\ \ \mbox{ on $(U\times U)\cap(\ol M\times\ol M)$} 
\end{equation}
or $F_k\equiv G_k\mod O(k^{-\infty})$ on $(U\times U)\cap(\ol M\times\ol M)$ 
if there is a $r_k(x,y)\in \cC^\infty(U\times U)$ with $r_k(x,y)=O(k^{-\infty})$ 
on $U\times U$ such that  \[r_k(x,y)|_{(U\times U)\cap(\ol M\times\ol M)}=F_k(x,y)-G_k(x,y),\:
\text{for $k\gg1$.}\]  
Let $m\in\mathbb R$. Let $U$ be an open set in $M'$. Let 
\begin{equation}\label{e-gue190813yydI}
S^m_{{\rm loc}}(1,(U\times U)\cap(\ol M\times\ol M))
\end{equation}
denote the space of restrictions to $(U\times U)\cap(\ol M\times\ol M)$ of elements in 
$S^m_{{\rm loc}}(1,U\times U)$.
Let
\[a_j\in S^{m_j}_{{\rm loc}}(1,(U\times U)\cap(\ol M\times\ol M)),\ \ j=0,1,2,\dots,\] 
with $m_j\searrow -\infty$, $j\To \infty$.
Then there exists
$a\in S^{m_0}_{{\rm loc}}(1,(U\times U)\cap(\ol M\times\ol M))$
such that for every $\ell\in\N$,
\[a-\sum^{\ell-1}_{j=0}a_j\in S^{m_\ell}_{{\rm loc}}(1,(U\times U)\cap
(\ol M\times\ol M)).\]
If $a$ and $a_j$ have the properties above, we write
\[a\sim\sum^\infty_{j=0}a_j \text{ in }
S^{m_0}_{{\rm loc}}(1,(U\times U)\cap(\ol M\times\ol M)).\]

\section{The Toeplitz operator $T_R$}\label{s-gue230523yyd}
Let $R$ be a first order partial differential operator on $M'$ 
such that $R$ is formally self-adjoint with respect to $(\,\cdot\,|\,\cdot\,)_{M'}$ and near $X$,
\begin{equation}
R=\frac{1}{2}((-iT)+(-iT)^*),
\end{equation} 
where $T$ is given by \eqref{e-gue190312scdqI} and $(-iT)^*$ is the formal adjoint of 
$-iT$ with respect to $(\,\cdot\,|\,\cdot\,)_{M'}$. 
Let $T_R$ the Toeplitz operator introduced in \eqref{e-gue230528yydq}.
The goal of this section is to prove the following.
\begin{theorem}\label{t-gue230525yyd}
The operator $T_R:\Dom(T_R)\subset L^2(M)\To L^2(M)$ is 
self-adjoint.
\end{theorem}

For the proof of the theorem~\ref{t-gue230525yyd} we need some preparation.
Let 
\[\ddbar^*_f: \Omega^{0,1}(M')\To\cC^\infty(M')\]
be the formal adjoint of $\ddbar$ with respect to $(\,\cdot\,|\,\cdot\,)_{M'}$, that is,
$(\,\ddbar f\,|\,h\,)_{M'}=(\,f\,|\,\ddbar^*_fh\,)_{M'}$, for any
$f\in\cC^\infty_c(M')$, $h\in\Omega^{0,1}(M')$. Let
\[\Box_f=\ddbar^*_f\,\ddbar: \cC^\infty(M')\To\cC^\infty(M')\]
denote the complex Laplace-Beltrami operator on functions. 
Let
\begin{equation}\label{e-gue171010yII}
P: \cC^\infty(X)\rightarrow\cC^\infty(\overline M)
\end{equation}
be the Poisson operator associated to $\Box_f$. The Poisson operator $P$ satisfies
\begin{equation}\label{e-gue171011II}
\begin{split}
&\Box_fPu=0,\ \  u\in\cC^\infty(X),\\
&\gamma Pu=u,\ \  u\in\cC^\infty(X),
\end{split}
\end{equation}
where $\gamma$ denotes the operator of restriction 
to the boundary $X$. It is known that $P$ extends continuously
\[P: H^s(X)\rightarrow H^{s+\frac{1}{2}}(\overline M),\ \ \forall s\in\mathbb R\]
(see~\cite[Page 29]{B71}). Let
\[P^*: \hat{\mathscr D}'(\overline M)\rightarrow\mathscr D'(X)\] be the operator defined by
\[(\,P^*u\,|\,v\,)_X=(\,u\,|\,Pv\,)_M,\ \ 
u\in\hat{\mathscr D}'(\overline M),\ \  v\in\cC^\infty(X),\]
where $\hat{\mathscr D}'(\overline M)$ denotes the space 
of continuous linear maps from $\cC^\infty(\ol M)$ to 
$\C$ with respect to $(\,\cdot\,|\,\cdot\,)_M$.
It is well-known (see~\cite[page 30]{B71}) that $P^*$ 
is continuous $P^*: H^s(\ol M)\rightarrow H^{s+\frac{1}{2}}(X)$
for every $s\in\R$ and
\[P^*: \cC^\infty(\overline M)\rightarrow\cC^\infty(X).\]
It is well-known that the operator
\[P^*P: \cC^\infty(X)\To\cC^\infty(X)\]
is a classical elliptic pseudodifferential operator of order $-1$ and invertible since $P$ is
injective~(see~\cite{B71}). Moreover, the operator
\[(P^*P)^{-1}: \cC^\infty(X)\To\cC^\infty(X)\]
is a classical elliptic pseudodifferential operator of order one.
We define a new inner product on $H^{-\frac{1}{2}}(X)$ as follows:
\begin{equation}\label{inner product}
[\,u\,|\,v\,]_X:=(\,Pu\,|\,Pv)_{M},\ \ u, v\in H^{-\frac{1}{2}}(X).
\end{equation}
For an operator $A$ on $H^{-\frac{1}{2}}(X)$
we denote by $A^\dagger$ the formal adjoint of $A$ 
with respect to the inner product $[\,\cdot\,|\,\cdot\,]_X$.

The next result shows that we can link the Bergman
projection with a certain approximate projector on the boundary $X$,
which is a Fourier integral operator.

\begin{theorem}[{\cite{Hsiao08}}]%[{\cite[Part II, (7.4), Proposition 7.5]{Hsiao08}}]
\label{t-gue230519yydI}
There exists a continuous operator 
\begin{equation}\label{eq:hats0}
\mathcal{S}:\cC^\infty(X)\To\cC^\infty(X), \quad 
\mathcal{S}\in L^0_{\frac{1}{2},\frac{1}{2}}(X)
\end{equation}
such that 
\begin{equation}\label{e-gue230519yyd}
B=P\mathcal{S}(P^*P)^{-1}P^*,\quad\text{on $\cC^\infty(\ol M)$}, 
\end{equation}
with the following properties,
\begin{equation}\label{eq:hats1}
\mathcal{S}^\dagger=\mathcal{S}, \quad 
\mathcal{S}^2=\mathcal{S}\:\:\text{on $\mathscr D'(X)$},
\end{equation}
and for any local coordinate patch $(D,x)$, we have 
\begin{equation}\label{e-gue230526yyd}
\mathcal{S}(x,y)\equiv\int_0^\infty e^{it\varphi(x,y)}s(x,y,t)dt
\quad\text{on $D\times D$}, 
\end{equation}
where $\varphi=\varphi_-\in\cC^\infty(D\times D)$ 
is the phase function $\varphi_-$ 
as in~\cite[Theorem 4.1]{HM14} satisfying 
\begin{equation}\label{e-gue140205IV}
\begin{split}
&\varphi\in \cC^\infty(D\times D),\ \ {\rm Im\,}\varphi(x, y)\geq0,\\
&\varphi(x, x)=0,\ \ \varphi(x, y)\neq0\ \ \mbox{if}\ \ x\neq y,\\
&d_x\varphi(x, y)\big|_{x=y}=-d_y\varphi(x, y)\big|_{x=y}=\omega_0(x), \\
&\varphi(x, y)=-\ol\varphi(y, x).
\end{split}
\end{equation}
and 
\begin{equation}\label{e-gue140205IVa}
\begin{split}
&s(x,y,t)\sim \sum_{j=0}^{+\infty}s_j(x,y)t^{n-1-j}\:\:
\text{in $S^{n-1}_{1,0}(D\times D\times\mathbb{R}_+)$},\\ 
&s_j(x,y)\in\cC^\infty(D\times D), j=0,1,\ldots,\\ 
&s_0(x,x)=\frac{1}{2}\pi^{-n}\det(\cL_x),\:\:\text{for all $x\in D_0$}.
\end{split}
\end{equation}
\end{theorem}
\begin{proof}
We recall here the construction from \cite{Hsiao08} 
for the convenience of the reader.
In order to link the Bergman projection to a boundary
operator we consider a version of the tangential Cauchy-Riemann
operator $\ddbar_b$ on $X$, denoted $\ddbar_{\beta}$,
expressed in terms of the Poisson extension operator, the $\ddbar$
operator and the restriction to the boundary, namely 
\[\ddbar_{\beta}:\Omega^{0,\star}(X)\to\Omega^{0,\star\,+1}(X),
\quad\ddbar_{\beta}=Q\gamma\ddbar P,
\] 
where $P$ is the Poisson operator,
$\gamma:\Omega^{0,\star}(\overline{M})\to\Omega^{0,\star}(X)$
is the restriction operator and 
$Q:H^{-1/2}(X,\Lambda^{0,\star}TM')\to
\ker(\ddbar\rho\wedge\cdot)^*\subset 
H^{-1/2}(X,\Lambda^{0,\star}TM')$
is the orthogonal projection,
cf.\ \cite[(5.1), p.\,103]{Hsiao08}.
Note that $Q$ is the operator $T$ in~\cite[(3.6), p.\,96]{Hsiao08} 
and $Q$ is the identity in degree zero.
The operator $\ddbar_{\beta}$ is a classical pseudo-differential
operator of order one on $X$, such that
$\ddbar_{\beta}=\ddbar_{b}+\text{l.o.t.}$ and
$\ddbar_{\beta}^2=0$.
The corresponding Laplace operator (cf.\ \cite[(5.6), p.\,104]{Hsiao08}),
\[\Box_{\beta}^{(\star)}=
\ddbar_{\beta}\,\ddbar_{\beta}^{\,\dagger}+
\ddbar_{\beta}^{\,\dagger}\,\ddbar_{\beta}:
\Omega^{0,\star}(X)\to\Omega^{0,\star}(X)
\] 
is a classical pseudo-differential
operator of order two on $X$, with the same principal symbol
and the same characteristic manifold
as the Kohn Laplacian $\Box_{b}=\ddbar_{b}\,\ddbar_{b}^{\,*}+
\ddbar_{b}^{\,*}\,\ddbar_{b}$,
\begin{equation}\label{eq:char1}
\Sigma=\left\{(x,t\omega_0(x))\in T^*X:x\in X,\:t\in\R\setminus\{0\}\right\}.
\end{equation}
We have 
\begin{equation}\label{eq:char2}
\Sigma=\Sigma^+\cup\Sigma^-,\:\:
\Sigma^+:=\{(x,t\omega_0(x))\in T^*X:x\in X,\:t>0\},\:\:\Sigma^-:=\Sigma\setminus\Sigma^+.
\end{equation}
Note that we use here a different sign convention than in \cite{Hsiao08},
where $\omega_0$ equals $d\rho\circ J$ (compare 
\cite[(1.9), p.\,84]{Hsiao08}, \eqref{eq:omega_0}),
thus we swap here the roles of $\Sigma^{+}$
and $\Sigma^{-}$ compared to \cite{Hsiao08}.

By Theorem \cite[Theorem 6.15, p.\,114]{Hsiao08} 
the operator $\Box_\beta^{(0)}$
acting in degree $q=0$ (that is on functions)
has a parametrix $A$
and an approximate projector $S$ 
(denoted $B_-$ in \cite{Hsiao08}) such that
\begin{equation}
\begin{split}
A\in L^{-1}_{\frac12,\frac12}(X),&\quad S\in L^{0}_{\frac12,\frac12}(X),\\
A\Box_\beta^{(0)}+S\equiv I,&\quad 
\Box_\beta^{(0)}A+S\equiv I,\\
S^2\equiv S,&\quad S^\dagger\equiv S,\\
\ddbar_\beta\,S\equiv0,&\quad \ddbar_\beta^{\,\dagger}\,S\equiv0.
\end{split}
\end{equation}
Morevover the wavefront set of the distribution kernel $S(\cdot,\cdot)$
of $\mathcal{S}$ is given by
\begin{equation}
\operatorname{WF}(S(\cdot,\cdot))=
\{(x,\xi,x,-\xi):(x,\xi)\in\Sigma^+\}\,.
\end{equation}
In \cite[(7.4), p.\,120]{Hsiao08} the operator 
$P\mathcal{S}Q(P^*P)^{-1}P^*$
is defined on $\Omega^{0,q}(\overline{M})$
and it is shown in \cite[Proposition 7.5]{Hsiao08} that its kernel equals
the Bergman kernel on $(0,q)$-forms up to a smooth form on 
$\overline{M}\times\overline{M}$. For $q=0$ the operator $Q$
is the identity so we have $B=PS(P^*P)^{-1}P^*+F$, where $F$
is smoothing. We set 
\begin{equation}\label{eq:mathcalS}
\mathcal{S}=S+(P^*P)^{-1}P^*FP.
\end{equation}
Then $\mathcal{S}^2=\mathcal{S}$, $\mathcal{S}^\dagger=\mathcal{S}$,
and $B=P\mathcal{S}(P^*P)^{-1}P^*$.
We have thus obtained \eqref{e-gue230519yyd} and \eqref{eq:hats1}.
The prpperties \eqref{e-gue230526yyd}, \eqref{e-gue140205IV}
and \eqref{e-gue230526yyd} follow from the corresponding properties of $S$.
\end{proof}
\begin{remark}
The operator $\mathcal{S}$ 
is a Toeplitz structure on $\Sigma^+$ 
in the sense of~\cite[Definition 2.10]{BG81}. 
\end{remark}

\begin{lemma}\label{l-gue230519yyd} 
For any $u\in\Dom(T_R)$ there exists $u_j\in\cC^\infty(\ol M)$, 
$j=1,2,\ldots$, such that 
$\lim_{j\To+\infty}u_j=Bu$ in $L^2(M)$ and 
$\lim_{j\To+\infty}BRBu_j=BRBu$ in $L^2(M)$.
\end{lemma}

\begin{proof}
Let $u\in\Dom(T_R)$. We may assume that $u=Bu$. Then, 
\[u=P(P^*P)^{-1}P^*u=P\mathcal{S}(P^*P)^{-1}P^*u.\] 
From \eqref{e-gue230519yyd}, we have 
\begin{equation}\label{e-gue230519yyda}
BRBu=P\mathcal{S}(P^*P)^{-1}P^*RP\mathcal{S}(P^*P)^{-1}P^*u
=P\mathcal{S}L\mathcal{S}(P^*P)^{-1}P^*u,
\end{equation}
where $L=(P^*P)^{-1}P^*RP$. It is straightforward to check that 
$L\in L^1_{{\rm cl\,}}(X)$ and 
\[\sigma^0_L(x,\omega_0(x))\neq0\] 
at every $x\in X$, where $\sigma^0_L$ denotes the principal symbol of $L$. Since 
$\mathcal{S}L\mathcal{S}(P^*P)^{-1}P^*u\in H^{-\frac{1}{2}}(X)$, 
we can repeat the proof of~\cite[Theorem 3.3]{HHMS23} and deduce that 
$(P^*P)^{-1}P^*u\in H^{\frac{1}{2}}(X)$. Let $v_j\in\cC^\infty(X)$, $j=1,2,\ldots$, 
$v_j\To(P^*P)^{-1}P^*u$ in $H^{\frac{1}{2}}(X)$ as $j\To+\infty$. 
Then, $u_j:=Pv_j\To P(P^*P)^{-1}P^*u=u$ in $H^{1}(\ol M)$ as $j\To+\infty$ and 
$BRBu_j\To BRBu$ in $L^2(\ol M)$ as $j\To+\infty$.
\end{proof} 

\begin{proof}[Proof of Theorem~\ref{t-gue230525yyd}]
Let $T_R^*: \Dom(T_R^*)\subset L^2(X)\To L^2(X)$ be the $L^2$ 
adjoint of $T_R$. Let $u\in\Dom(T_R)$. From Lemma~\ref{l-gue230519yyd}, 
for every $v\in\Dom(T_R)$, we have 
\begin{equation}\label{e-gue230519yyds}
(\,u\,|\,Av\,)_M=(\,Bu\,|\,Av\,)_M=\lim_{j\To+\infty}(\,Bu_j\,|\,BRBv_j\,)_M,
\end{equation}
where $u_j, v_j\in\cC^\infty(\ol M)$, $j=1,2,\ldots$, such that 
$\lim_{j\To+\infty}u_j=Bu$ in $L^2(M)$, $\lim_{j\To+\infty}v_j=Bv$ in $L^2(M)$, 
$\lim_{j\To+\infty}BRBu_j=BRBu$ in $L^2(M)$ and $\lim_{j\To+\infty}BRBv_j=BRBv$ 
in $L^2(M)$. From \eqref{e-gue230519yyds} and since $R\rho=0$ on $X$, 
we can integrate by parts and deduce that 
\[(\,u\,|\,T_Rv\,)_M=(\,T_Ru\,|\,v\,)_M,\ \ \mbox{for every $v\in\Dom(T_R)$}.\]
Thus, $u\in\Dom(T_R^*)$ and $T_R^*u=T_Ru$. 
Let $u\in\Dom(T_R^*)$. Since $\cC^\infty(\ol M)\subset\Dom(T_R)$, 
we deduce that there is a constant $C>0$ such that 
\[\abs{(\,u\,|\,BRBv\,)_M}\leq C\norm{v}_M,\ \ \mbox{for every $v\in\cC^\infty(\ol M)$}.\]
Thus, $BRBu\in L^2(M)$ and hence $u\in\Dom(T_R)$.
\end{proof} 

\section{Asymptotic expansion of $\chi_k(T_R)$}\label{s-gue230527yyd}

In this Section we will reduce the study of the
Toeplitz operator $T_R$ to the study of 
a Toeplitz operator $\mathcal{T}_{\mathcal{R}}$ on the boundary
$X$ and apply results from \cite{HHMS23} in order
to prove Theorem \ref{t-gue230528yydmp}.
The Toeplitz operator on the boundary is defined by 
\begin{equation}\label{e-gue230527yyda}
\mathcal{T}_{\mathcal{R}}:=
\mathcal{S}\mathcal{R}\mathcal{S}:\cC^\infty(X)\To\cC^\infty(X), 
\end{equation}
where $\mathcal{S}$ is as in Theorem~\ref{t-gue230519yydI} and 
\begin{equation}\label{e-gue230527yydab}
\mathcal{R}:=(P^*P)^{-1}P^*RP\in L^1_{{\rm cl\,}}(X). 
\end{equation}
Note that by \eqref{e-gue230519yyda} we have
\begin{equation}\label{e-gue230527yydaba}
T_R=P(\mathcal{S}\mathcal{R}\mathcal{S})(P^*P)^{-1}P^*=
P\mathcal{T}_{\mathcal{R}}(P^*P)^{-1}P^*. 
\end{equation}
We extend 
$\mathcal{T}_{\mathcal{R}}$ to $H^{-\frac{1}{2}}(X)$: 
\begin{equation}\label{e-gue230527yydaba1}
\begin{split}
&\mathcal{T}_{\mathcal{R}}: \Dom(\mathcal{T}_{\mathcal{R}})\subset 
H^{-\frac{1}{2}}(X)\To H^{-\frac{1}{2}}(X),\\
&\Dom(\mathcal{T}_{\mathcal{R}})=\set{u\in H^{-\frac{1}{2}}(X);\, 
\mathcal{S}\mathcal{R}\mathcal{S}u\in H^{-\frac{1}{2}}(X)}.
\end{split}
\end{equation}
The operator $\mathcal{S}$ is a Toeplitz structure (generalized Szeg\H{o}
projector) in the sense 
of~\cite[Definition 2.10]{BG81}.
Let $\operatorname{Im}(\mathcal{S})$
be the image of $\mathcal{S}$ in $L^2(X)$.
By~\cite[Proposition 2.14]{BG81} the spectrum of
the operator 
$\mathcal{T}_{\mathcal{R}}|_{\operatorname{Im}(\mathcal{S})}:
\operatorname{Im}(\mathcal{S})
\to\operatorname{Im}(\mathcal{S})$ consists only of isolated
eigenvalues of finite multiplicity, is bounded from below
and has only $+\infty$ as a point of accumulation.
We have $\operatorname{Spec}(\mathcal{T}_{\mathcal{R}})\setminus\{0\}=
\operatorname{Spec}(\mathcal{T}_{\mathcal{R}}|_{\operatorname{Im}(\mathcal{S})})
\setminus\{0\}$ and the restrictions to $\R\setminus\{0\}$ of spectral measures
of these operators coincide.
We conclude that the operator 
$\mathcal{T}_{\mathcal{R}}$ in \eqref{e-gue230527yydaba1} 
is self-adjoint with respect to $[\,\cdot\,|\,\cdot\,]_X$ and 
its spectrum consists only of isolated
eigenvalues, is bounded from below
and has only $+\infty$ as a point of accumulation.
Moreover, for every 
$\lambda\in\operatorname{Spec}(\mathcal{T}_{\mathcal{R}})$, 
$\lambda\neq0$, the eigenspace 
\[E_\lambda=\set{u\in\Dom(\mathcal{T}_{\mathcal{R}}):
\mathcal{T}_{\mathcal{R}}u=\lambda u}\]
is a finite dimensional subspace of $\cC^\infty(X)$.
\begin{remark}\label{rem:chi}
The kernel of $\mathcal{T}_{\mathcal{R}}$ contains
the kernel of $\mathcal{S}$, so in order to avoid the zero
eigenvalue we consider the operator $\chi_k(\mathcal{T}_{\mathcal{R}})$
associated to a function $\chi$ with support in $(0,+\infty)$.
In this way the image of $\chi_k(\mathcal{T}_{\mathcal{R}})$
is contained in $\operatorname{Im}(\mathcal{S})$.
\end{remark}
%===
\begin{lemma}\label{l-gue230520yyd}
For every $z\in\mathbb C$, $z\notin\mathbb R$, we have 
\begin{equation}\label{e-gue230527yydf}
(z-T_R)^{-1}B=P(z-\mathcal{T}_{\mathcal{R}})^{-1}\mathcal{S}(P^*P)^{-1}P^*.
\end{equation}
\end{lemma}

\begin{proof}
From \eqref{e-gue230519yyd} and \eqref{e-gue230519yyda} we have 
\begin{equation}\label{e-gue230520yyd}
(z-T_R)B=P(z-\mathcal{S}\mathcal{R}\mathcal{S})(P^*P)^{-1}P^*B=P(z-\mathcal{T}_{\mathcal{R}})\mathcal{S}(P^*P)^{-1}P^*.
\end{equation}
From \eqref{e-gue230520yyd}, we have 
\[P(z-\mathcal{T}_{\mathcal{R}})^{-1}(P^*P)^{-1}P^*(z-T_R)B=B.\]
Thus,
\[\begin{split}
    (z-T_R)^{-1}B&=P(z-\mathcal{T}_{\mathcal{R}})^{-1}(P^*P)^{-1}P^*B\\
    &=P(z-\mathcal{T}_{\mathcal{R}})^{-1}(P^*P)^{-1}P^*P\mathcal{S}(P^*P)^{-1}P^*\\
    &=P(z-\mathcal{T}_{\mathcal{R}})^{-1}\mathcal{S}(P^*P)^{-1}P^*.\end{split}\]
    The lemma follows. 
\end{proof}

\begin{lemma}\label{l-gue230527ycd}
We have 
\begin{equation}\label{e-gue230527ycda}
\chi_k(T_R)=P\chi_k(\mathcal{T}_{\mathcal{R}})(P^*P)^{-1}P^*.
\end{equation}
\end{lemma}

\begin{proof}
From the Helffer-Sj\"ostrand formula \cite[\S 8]{DiSj99} and 
\eqref{e-gue230527yydf}, we have 
\[\begin{split}
\chi_k(T_R)&=\frac{1}{2\pi i}\int_{\mathbb C}\frac{\pr\tilde\chi_k}{\pr\ol z}(z)(z-T_R)^{-1}dzd\ol z\\
&=\frac{1}{2\pi i}\int_{\mathbb C}\frac{\pr\tilde\chi_k}{\pr\ol z}(z)P(z-\mathcal{T}_{\mathcal{R}})^{-1}\mathcal{S}(P^*P)^{-1}P^*dzd\ol z\\
&=P\chi_k(\mathcal{T}_{\mathcal{R}})(P^*P)^{-1}P^*,
\end{split}\]
where $\tilde\chi_k$ denotes an almost analytic extension of $\chi_k$. The lemma follows. 
\end{proof}

\begin{corollary}\label{t-gue230527ycd}
We have 
\begin{equation}\label{e-gue230527ycds}
\chi_k(T_R)(x,y)\in\cC^\infty(\ol M\times\ol M).
\end{equation}
%Let $\tau, \hat\tau\in\cC^\infty(\ol M)$, 
%$\supp\tau\cap\supp\hat\tau=\emptyset$. We have 
%\begin{equation}\label{e-gue230527yydz}
%\tau\chi_k(T_R)\hat\tau\equiv0\mod O(k^{-\infty})\ \ 
%\mbox{on $\ol M\times\ol M$}. 
%\end{equation}
\end{corollary}
\begin{proof}
This follows from \eqref{e-gue230527ycda} and from the fact
that $\chi_k(\mathcal{T}_{\mathcal{R}})\in\cC^\infty(X\times X)$.
%we get \eqref{e-gue230527ycds}. 
\end{proof}
We need the following variant of \cite[Theorem 1.1]{HHMS23}.
%---
\begin{theorem}
\label{thm:ExpansionMain}
Let $(X,HX,J)$ be an orientable 
compact strictly pseudoconvex Cauchy-Riemann manifold
of dimension $2n-1$, $n\geq2$.
%such that the Kohn Laplacian on $X$ has closed range in 
%$L^2(X)$. 
We consider:

(a) A Riemannian metric $g^{TX}$ compatible with $J$,
with volume form $dv_X$ and the associated $L^2$-space 
$L^2(X)=L^2(X,dv_X)$.
 
(b) A contact form $\omega_0$ on $X$ such that the Levi form
$\mathcal{L}=\frac12d\omega_0(\cdot,J\cdot)$ is positive definite.
We denote by $dv_{\omega_0}=\omega_0\wedge(d\omega_0)^{n-1}$.

(c) An operator $\mathcal{S}:\cC^\infty(X)\To\cC^\infty(X)$, 
satisfying \eqref{eq:hats0} and 
\eqref{eq:hats1}-%, \eqref{e-gue230526yyd},
\eqref{e-gue140205IVa}.

(d) For a 
formally self-adjoint first order pseudodifferential operator 
$Q\in L^1_\mathrm{cl}(X)$ we consider the Toeplitz operator 
$\mathcal{T}_Q=\mathcal{S}Q\mathcal{S}:L^2(X)\to L^2(X)$.

Let $(D,x)$ be any coordinates patch and let 
$\varphi:D\times D\to\C$ be the phase function satisfying 
\eqref{e-gue230526yyd} and \eqref{e-gue140205IV}.
Then for any formally self-adjoint first order pseudodifferential operator 
$Q\in L^1_\mathrm{cl}(X)$ whose symbol $\sigma_Q$ satisfies
$\sigma_Q(\omega_0)>0$ on $X$,  
and for any $\chi\in\cC^\infty_c((0,+\infty))$, $\chi\not\equiv 0$,
the Schwartz kernel of $\chi_k(T_Q)$, 
$\chi_{k}(\lambda):=\chi\left(k^{-1}\lambda\right)$,
can be represented for $k$ large by
\begin{equation}
\label{eq:asymptotic expansion of chi_k(T_Q)}
\chi_k(\mathcal{T}_Q)(x,y)=\int_0^{+\infty} 
e^{ikt\varphi(x,y)}{A}(x,y,t,k)dt+O\left(k^{-\infty}\right)~\text{on}~D\times D,
\end{equation}
where ${A}(x,y,t,k)\in S^{n}_{\mathrm{loc}}
(1;D\times D\times{\R}_+)$,
\begin{equation}
\label{Eq:LeadingTermMainThm}
\begin{split}
&{A}(x,y,t,k)\sim\sum_{j=0}^{+\infty} {A}_{j}(x,y,t)k^{n-j}~
\mathrm{in}~S^{n+1}_{\mathrm{loc}}(1;D\times D\times{\R}_+),\\
&A_j(x,y,t)\in\mathscr{C}^\infty(D\times D\times{\R}_+),~j=0,1,2,\ldots,\\
&{A}_{0}(x,x,t)=\frac{1}{2\pi ^{n}}
\frac{dv_{\omega_0}}{dv_X}(x)\,
\chi(t\sigma_Q(\omega_0(x)))\,t^{n-1}\not\equiv 0,
\end{split}
\end{equation}
and for some compact interval $I\Subset\R_+$,
\begin{equation}
\begin{split}
\supp_t A(x,y,t,k),~\supp_t A_j(x,y,t)\subset I,\ \ j=0,1,2,\ldots.
\end{split}
\end{equation}
Moreover, for any $\tau_1,\tau_2\in\cC^\infty(X)$ 
such that $\supp(\tau_1)\cap\supp(\tau_2)=\emptyset$, 
we have
\begin{equation}
\label{Eq:FarAwayDiagonalMainThm}
\tau_1\chi_k(T_P)\tau_2=O\left(k^{-\infty}\right).
\end{equation}
\end{theorem} 
%---
The proof of Theorem \ref{thm:ExpansionMain} is completely
analogous to the proof of \cite[Theorem 1.1]{HHMS23} on account
of the structure of $\mathcal{S}$ as a Fourier integral operator 
given in Theorem \ref{t-gue230519yydI}.
%---
\begin{proof}[Proof of Theorem~\ref{t-gue230528yydmp}]
We will apply Theorem \ref{thm:ExpansionMain}
for $X=\partial{M}$ as in Theorem \ref{t-gue230528yydmp}.
The metric $g^{TX}$ in (a) is induced by the metric $g^{TM'}$
and the contact form $\omega_0$ in (b) is given by 
\eqref{eq:omega_0}-\eqref{eq:2.12b}.
The operator $\mathcal{S}$ in (c) is the operator constructed
in Theorem \ref{t-gue230519yydI}, which in particular
fulfills \eqref{e-gue230519yyd}. Moreover, we apply 
Theorem \ref{thm:ExpansionMain} for $Q=\mathcal{R}$
given by \eqref{e-gue230527yydab}.
In this situation, we have 
\begin{equation}
\frac{dv_{\omega_0}}{dv_X}(x)=\det(\cL_x),\quad
\sigma_{\mathcal{R}}(\omega_0)=\omega_0(T).
\end{equation}
By \eqref{e-gue190312scdqII} and \eqref{eq:omegaT} we have
$\sigma_{\mathcal{R}}(\omega_0)=\omega_0(T)>0$ on $X$.

We first prove (i).
Let $\tau, \hat\tau\in\cC^\infty(\ol M)$, 
$\supp\tau\cap\supp\hat\tau=\emptyset$. We have 
\begin{equation}\label{e-gue230527ycdp}
\begin{split}
&\tau\chi_k(T_R)\hat\tau\\
&=\tau P\chi_k(\mathcal{T}_{\mathcal{R}})(P^*P)^{-1}P^*\hat\tau\\
&=\tau P\tau_1\chi_k(\mathcal{T}_{\mathcal{R}})\hat\tau_1(P^*P)^{-1}P^*\hat\tau
+\tau P(1-\tau_1)\chi_k(\mathcal{T}_{\mathcal{R}})\hat\tau_1(P^*P)^{-1}P^*\hat\tau\\
&\quad+\tau P\chi_k(\mathcal{T}_{\mathcal{R}})(1-\hat\tau_1)(P^*P)^{-1}P^*\hat\tau,
\end{split}
\end{equation}
where $\tau_1, \hat\tau_1\in\cC^\infty(X)$, $\supp\tau_1
\cap\supp\hat\tau_1=\emptyset$, $\supp\tau\cap\supp(1-\tau_1)
=\emptyset$, 
\[\supp(1-\hat\tau_1)\cap\supp\hat\tau=\emptyset.\] 
We apply now Theorem~\ref{thm:ExpansionMain} 
%\cite[Theorem 1.1]{HHMS23}, 
for the operator $Q=\mathcal{R}$ and 
we see that 
$\tau_1\chi_k(\mathcal{T}_{\mathcal{R}})\hat\tau_1=O(k^{-\infty})$ 
and hence 
\begin{equation}\label{e-gue230527ycdr}
\tau P\tau_1\chi_k(\mathcal{T}_{\mathcal{R}})\hat\tau_1(P^*P)^{-1}P^*\hat\tau
=O(k^{-\infty})\ \ \mbox{on $\ol M\times\ol M$}.
\end{equation}
From~\cite[Lemma 4.1]{HM19}, we see that 
\begin{equation}\label{e-gue230527ycdq}
(\tau P(1-\tau_1))(x,y)\in\cC^\infty(\ol M\times X),
\end{equation}
where $(\tau P(1-\tau_1))(x,y)$ denotes the distribution kernel 
of $\tau P(1-\tau_1)$. From \eqref{e-gue230527ycdq}, 
we can repeat the proof of~\cite[Theorem 4.6]{HHMS23} with 
minor changes and deduce that  
\begin{equation}\label{e-gue230527ycdu}
\tau P(1-\tau_1)\chi_k(\mathcal{T}_{\mathcal{R}})
\hat\tau_1(P^*P)^{-1}P^*\hat\tau=O(k^{-\infty})\ \ 
\mbox{on $\ol M\times\ol M$}.
\end{equation} 
Similarly, from~\cite[Lemma 4.2]{HM19}, we see that 
\begin{equation}\label{e-gue230528yyd}
((1-\hat\tau_1)(P^*P)^{-1}P^*\hat\tau)(x,y)\in\cC^\infty(\ol X\times\ol M),
\end{equation}
where $((1-\hat\tau_1)(P^*P)^{-1}P^*\hat\tau)(x,y)$ denotes the distribution kernel of $(1-\hat\tau_1)(P^*P)^{-1}P^*\hat\tau$. From \eqref{e-gue230528yyd}, we can repeat the proof of~\cite[Theorem 4.6]{HHMS23} with 
minor changes and deduce that  
\begin{equation}\label{e-gue230528yydI}
\tau P\chi_k(\mathcal{T}_{\mathcal{R}})(1-\hat\tau_1)(P^*P)^{-1}P^*
\hat\tau=O(k^{-\infty})\ \ \mbox{on $\ol M\times\ol M$}.
\end{equation} 

From \eqref{e-gue230527ycdp}, \eqref{e-gue230527ycdr}, 
\eqref{e-gue230527ycdu} and \eqref{e-gue230528yydI}, we get 
\eqref{e-gue230527yydzz}.
%\eqref{e-gue230527yydz}.

We prove now (ii) and (iii).
Fix $p\in\ol M$. We first assume that $p\notin X$ 
and let $U$ be an open set of $p$ 
with $U\cap X=\emptyset$. Let $\tau\in\cC^\infty_c(U)$. 
Since 
$(\tau P)(x,y)\in\cC^\infty(\ol M\times X)$, we can repeat 
the proof of~\cite[Theorem 4.6]{HHMS23} with minor changes and get 
\begin{equation}\label{e-gue230528yydm}
\tau P\chi_k(\mathcal{T}_{\mathcal{R}})(P^*P)^{-1}P^*
=O(k^{-\infty})\:\:\text{on $\ol M\times\ol M$}.
\end{equation}
From \eqref{e-gue230527ycda} and \eqref{e-gue230528yydm}, 
we get \eqref{e-gue230528ycdz}. %\eqref{e-gue230528ycd}. 

Now, assume that $p\in X$ and let $U$ be an open local 
coordinate patch of $p$ in $M'$. Let $D:=U\cap X$. 
We can repeat the proof of~\cite[Theorem 1.1]{HHMS23} 
(in the situation of Theorem \ref{thm:ExpansionMain}) and deduce 
\begin{equation}
\label{e-gue230528yydn}
\chi_k(\mathcal{T}_{\mathcal{R}})(x,y)=\int_0^{+\infty} 
e^{ikt\varphi(x,y)}a(x,y,t,k)dt+O\left(k^{-\infty}\right)
\quad\text{on $D\times D$},
\end{equation}
where $a(x,y,t,k)\in S^{n}_{\mathrm{loc}}
(1;D\times D\times{\mathbb R}_+)$,
\[
\begin{split}
&a(x,y,t,k)\sim\sum_{j=0}^\infty a_{j}(x,y,t)k^{n-j}~
\mathrm{in}~S^{n}_{\mathrm{loc}}(1;D\times D\times{\mathbb R}_+),\\
&a_j(x,y,t)\in\mathscr{C}^\infty(D\times D\times{\mathbb R}_+),~j=0,1,2,\ldots,\\
&a_{0}(x,x,t)=\frac{1}{2\pi ^{n}}
\det(\cL_x)\,\chi(t\omega_0(T(x)))t^{n-1}\not\equiv 0,
\end{split}
\]
and for some compact interval $I\Subset\mathbb R_+$,
\[
\begin{split}
\supp_t a(x,y,t,k),~\supp_t a_j(x,y,t)\subset I,\ \ j=0,1,2,\ldots.
\end{split}\]
From \eqref{e-gue230528yydn}, we can repeat the 
WKB procedure in~\cite[Part II,Proposition 7.8, Theorem 7.9]{Hsiao08} and get 
\eqref{e-gue230528ycdsz}.
%\eqref{e-gue230528ycds}.
\end{proof}

\begin{proof}[Proof of Corollary \ref{cor:trace}]
The asymptotics $\chi_k(T_R)(x,x)=O(k^{-\infty})$, $k\to\infty$,
on $M$ from \eqref{eq:tr1} follow immediately from 
\eqref{e-gue230528ycdz}.
Let $p\in X$ be fixed and consider local coordinates near $p$ on $M'$
of the form $z=(x_1,\ldots,x_{2n-1},\rho)$, where
$x=(x_1,\ldots,x_{2n-1})$ are local coordinates on $X$ near 
$p$ with $x(p)=0$ and the phase function $\psi$ 
in \eqref{e-gue230528ycdsz} has the form \eqref{e-gue230319ycdaIm}.
 In this local chart we have near $(p,p)$,
\begin{equation}\label{e-gue230319ycdaIm2}
i\Psi(z,z)=2\rho(z)\big(1+O(\abs{z}\big)+O(\abs{z}^3).
\end{equation}
By \eqref{e-gue230528ycdsz} we have
\begin{equation}\label{eq:1.8d}
\chi_k(T_P)(z,z)=\sum_{j=0}^{\infty} k^{n+1-j}\int_0^{+\infty}
e^{ikt\psi(z,z)}b_j(z,z,t)\,dt+O(k^{-\infty})
\end{equation}
Since $\Psi(x,x)=0$ for $x\in X$ this yields the asymptotic expansion
\eqref{eq:1.8} with the coefficients \eqref{eq:1.8a}.
The expression \eqref{eq:1.8b} of $b_0(x)$ follows from
\eqref{e-gue230528ycdtz}. We have $b_0(x)>0$ for every $x\in X$.
Note also the exponential decay of the integrands in 
\eqref{eq:1.8d} for $z\in M$ near $p$ due to 
\eqref{e-gue230319ycdaIm2} and on account of $\rho(z)<0$.

The trace of the operator $\chi_k(T_P)$
is given by 
\begin{equation}\label{eq:trace}
\begin{split}
\operatorname{Tr}\chi_k(T_P)&=\int_M\chi_k(T_P)(z,z)\,dv_{M'}(z)\\[2pt]
&=\int_{\{\rho<\varepsilon\}}\chi_k(T_P)(z,z)\,dv_{M'}(z)+
\int_{\{\varepsilon\leq\rho<0\}}\chi_k(T_P)(z,z)\,dv_{M'}(z)\\[2pt]
&=:I_1(k)+I_2(k).
\end{split}
\end{equation}
where $\varepsilon<0$ is chosen small enough.
We have $I_1(k)=O(k^{-\infty})$ by \eqref{eq:tr1}.
By using \eqref{e-gue230319ycdaIm2} and \eqref{eq:1.8d}
and the fact that $2k\int_\varepsilon^0e^{2k\rho}d\rho\to1$ as $k\to\infty$,
we obtain that there exist $C_1,C_2>0$ such that
$C_1k^n\leq I_2(k)\leq C_2k^n$ for $k$ large enough.
\end{proof}
\begin{remark}
It is interesting to compare the result of Corollary \ref{cor:trace}
to the corresponding result
regarding Toeplitz operators on the boundary $X$
(see also \cite[Corollary 1.2]{HHMS23}).
By Theorem \ref{thm:ExpansionMain} we have for the
operator $\mathcal{T}_{\mathcal{R}}$
from \eqref{e-gue230527yyda},
\begin{equation}\label{eq:1.8bf}
\chi_k(\mathcal{T}_{\mathcal{R}})(x,x)=
\sum_{j=0}^{\infty} A_j(x)k^{n-j}~
\text{in $S^{n+1}_{\rm loc}(1;X)$ on $X$},
\end{equation}
where 
\begin{equation}\label{eq:1.8af}
A_j(x)=\int_0^{+\infty}A_j(x,x,t)dt,\quad j\in\N_0,
\end{equation}
with $A_j(x,x,t)$ as in \eqref{Eq:LeadingTermMainThm}, and
\begin{equation}\label{eq:1.8ba}
A_0(x)=\frac{1}{2\pi ^{n}}\det(\cL_x)\int_0^{+\infty}
\chi(t\omega_0(T(x)))t^{n-1}dt.
\end{equation}
Moreover,
\begin{equation}\label{eq:1.9a}
\operatorname{Tr}\chi_k(\mathcal{T}_{\mathcal{R}})=
\frac{k^{n}}{2\pi^{n}}
\int_X\int_{0}^{+\infty}\det(\cL_x)\chi(t\omega_0(x))\,t^{n-1}
dt+O(k^{n-1})\,.
\end{equation}
We see that $\chi_k(T_{R})(x,x)$ and
$\chi_k(\mathcal{T}_{\mathcal{R}})(x,x)$ have an asymptotic expansion on
the boundary $X$, the former with leading term of order $k^{n+1}$, the latter
of order $k^n$.
On the other hand both traces $\operatorname{Tr}\chi_k(T_{R})$ and
$\operatorname{Tr}\chi_k(\mathcal{T}_{\mathcal{R}})$ have growth of order
$k^n$ as $k\to+\infty$.
\end{remark}

\begin{remark}
If we do not normalize the definition function $\rho$ such that
$|d\rho|=1$, Theorem \ref{t-gue230528yydmp} holds with the same
proof, but we need to take
$\omega_0=-J\circ d(\rho/|d\rho|)$. With this $\omega_0$ 
the leading term $b_{0}(x,x,t)$ has the same formula as in
\eqref{e-gue230528ycdtz}.
\end{remark}

\bibliographystyle{plain}

\begin{thebibliography}{999}

\bibitem{Berg32} 
S.~Bergman, 
\emph{\"Uber die Kernfunktion eines Bereiches und ihr Verhalten am Rande. I.}, 
J.\ Reine Angew.\ Math.\ \textbf{169}, 1--42 (1932).


\bibitem {B71} L.~Boutet de Monvel,
\emph{Boundary problems for pseudo-differential operators}, Acta Math.\
\textbf{126} (1971), 11--51.

\bibitem{BG81} L. Boutet de Monvel and V. Guillemin, 
\emph{The spectral theory of Toeplitz operators}, 
Ann.\ of Math.\ Stud., vol 99, Princeton Univ.\ Press, Princeton, NJ, 1981. v+161 pp.
	
\bibitem{BS75}
L.~Boutet de Monvel and J.~Sj{\"o}strand,
\emph{{Sur la singularit\'{e} des noyaux de Bergman et de Szeg\H{o}}.} 
{Journ{\'e}es {\'E}quations aux D{\'e}riv{\'e}es Partielles} 
(1975) 123-164. Ast\'{e}risque, No. 34-35.
 
\bibitem{DLM06}
X.~Dai, K.~Liu and X.~Ma, 
\emph{{On the asymptotic expansion of Bergman kernel.}} 
J. Differential Geom. \textbf{72} (2006), 1--41. 

\bibitem{DiSj99}
M.\ Dimassi and J.\ Sj\"ostrand,
\emph{{Spectral asymptotics in the semi-classical limit,}}
London Math. Soc. Lecture Note Ser., 268
Cambridge University Press, Cambridge, 1999. xii+227 pp.

 \bibitem{Fer74} 
C.~Fefferman, \emph{The {Bergman} kernel and biholomorphic 
mappings of pseudoconvex domains}, Invent. Math. \textbf{26} (1974), 1--65.

\bibitem{HHMS23}
H.~Herrmann, C.-Y.~Hsiao, G.~Marinescu and W.-C.~Shen, 
\emph{Semi-classical spectral asymptotics of Toeplitz operators
on CR manifolds}, arXiv.org: 2303.17319. 

\bibitem{Hor65}
L.~H\"{o}rmander, \emph{$L^2$ estimates and existence 
theorems for the $\ddbar$ operator}, Acta Math.\ \textbf{113}, (1965), 89--152. 

 \bibitem{Hsiao08}
	C.-Y.~Hsiao.
	\emph{{Projections in several complex variables}}, M\'emoires de la Soci\'et\'e
	Math\'ematique de France, 123 (2010), 131 pages.

 \bibitem{HM14CAG}
	C.-Y.~ Hsiao and G.~Marinescu. \newblock{\em Asymptotics of
		spectral function of lower energy forms and Bergman kernel of
		semi-positive and big line bundles.}
	Comm. Anal. Geom. \textbf{22} (2014), no. 1, 1--108.
	
	\bibitem{HM17CPDE}
	C.-Y.~Hsiao and G.~Marinescu, 
\emph{{Berezin-Toeplitz quantization for lower energy forms,}}
Comm. Partial Differential Equations \textbf{42} (2017), no.\ 6, 
895--942.
	

\bibitem{HM14} C.-Y.~Hsiao and G.~Marinescu, 
\emph{On the singularities of the Szeg\H{o} 
projections on lower energy forms}, J.\ Differential Geom.\
\textbf{107} (2017), no.\ 1, 83--155.

\bibitem{HM19} C.-Y. Hsiao and G Marinescu, 
\emph{On the singularities of the Bergman projections for 
lower energy forms on complex manifolds with boundary}, arXiv: 1911.10928. 


\bibitem{HMM} C.-Y. Hsiao, X. Ma and G. Marinescu, 
\emph{Geometric quantization on CR manifolds}, 
arXiv: 1906.05627, to appear at Commun.\ Contemp. Math., 
DOI: 10.1142/s0219199722500742.


\bibitem{HS22}
C.-Y. Hsiao and N. Savale,
\emph{Bergman-Szeg\H{o} kernel asymptotics in weakly pseudoconvex finite type cases},
J. Reine Angew. Math. \textbf{791} (2022), 173--223.

\bibitem{MM07}
	X.~Ma and G.~Marinescu.
	\emph{Holomorphic {M}orse inequalities and {B}ergman kernels.} Progress in  Mathematics, vol. 254, Birkh\"auser, Basel, 2007, 422 pp.
	
\end{thebibliography}

\end{document}